\documentclass[11pt,reqno]{amsart}
\usepackage{amsmath,amssymb,tikz-cd,stmaryrd,tikz-cd,enumitem,mathtools,microtype,hyperref,amsthm,array,booktabs,float}
\usepackage{cleveref}
\crefname{subsection}{subsection}{subsections}
\usepackage[hyphenbreaks]{breakurl}

\usepackage[frozencache,cachedir=minted-cache]{minted}
\theoremstyle{definition}
\newtheorem{theorem}{Theorem}[section]
\newtheorem*{theorem*}{Theorem}
\newtheorem{lemma}[theorem]{Lemma}
\newtheorem{proposition}[theorem]{Proposition}
\newtheorem{corollary}[theorem]{Corollary}

\newtheorem{example}[theorem]{Example}
\newtheorem{definition}[theorem]{Definition}
\newtheorem{remark}[theorem]{Remark}

\newtheorem*{question*}{Question}
\newtheorem*{remark*}{Remark}

\DeclareMathOperator{\lcm}{lcm}
\DeclareMathOperator{\LCLM}{LCLM}
\DeclareMathOperator{\Frac}{Frac}
\DeclareMathOperator{\SL}{SL}
\DeclareMathOperator{\GL}{GL}
\DeclareMathOperator{\PSL}{PSL}
\DeclareMathOperator{\Gal}{Gal}

\DeclareMathOperator{\Hom}{Hom}
\DeclareMathOperator{\spann}{span}
\DeclareMathOperator{\End}{End}
\DeclareMathOperator{\round}{round}
\DeclareMathOperator{\BesselK}{BesselK}
\DeclareMathOperator{\BesselI}{BesselI}

\newcommand{\cs}{\mathbin{\circledS}}
\newcommand{\pda}{\mathord{\downarrow}}
\newcommand{\pua}{\mathord{\uparrow}}
\definecolor{c1}{RGB}{203, 75, 48}

\title{Solving order 3 difference equations}

\makeatletter
\newcommand\footnoteref[1]{\protected@xdef\@thefnmark{\ref{#1}}\@footnotemark}
\makeatother

\makeatletter
\patchcmd{\@setaddresses}{\indent}{\noindent}{}{}
\patchcmd{\@setaddresses}{\indent}{\noindent}{}{}
\patchcmd{\@setaddresses}{\indent}{\noindent}{}{}
\patchcmd{\@setaddresses}{\indent}{\noindent}{}{}
\makeatother

\author{Heba Bou KaedBey}
\author{Mark van Hoeij}
\author{Man Cheung Tsui}
\address{Department of Mathematics\\ Florida State University\\ Tallahassee, FL, USA}
\email{hb20@fsu.edu\\ hoeij@math.fsu.edu\\ manctsui@gmail.com}


\begin{document}

\begin{abstract}
    We classify order $3$ linear difference operators over $\mathbb{C}(x)$ that are solvable in terms of lower order difference operators. 
    To prove this result, we introduce the notion of absolute irreducibility for difference modules, and classify (for arbitrary order) modules that are irreducible but not absolutely irreducible.
\end{abstract}
\maketitle

\section{Introduction}
\label{section:introduction}

The paper \cite{hendricks1999solving} determined \emph{Liouvillian solutions}
of difference operator. Liouvillian sequences are generated by solutions of order $1$ operators
through repeated applications of difference ring operations, indefinite summation, and interlacing of sequences.
By defining \emph{$2$-expressible sequences} 
to be generated in this way from solutions of order $2$ operators, we can state our main result as follows.

\begin{theorem*}[\Cref{thm:general-classification} restated]
    Let $L$ be an order $3$ difference operator over $\mathbb{C}(x)$.
    If $L$ is \emph{$2$-solvable} (Definition~\ref{solv2}), 
    then at least one of the following holds.
    \begin{enumerate}[label=(\alph*)]
        \item $L$ admits a nontrivial factorization over $\mathbb{C}(x)$.
        \item $L$ is gauge equivalent to $\Phi^{3}+f$ for some $f\in \mathbb{C}(x)$.
        \item $L$ is gauge equivalent to $L_{2}^{\cs 2}\cs L_{1}$ for some operators $L_{1}$ and $L_{2}$ of orders $1$ and $2$, respectively, over $\mathbb{C}(x)$.
    \end{enumerate}
\end{theorem*}

Each case of this classification can be solved algorithmically \cite{bou2024solving}.\footnote{Our paper \cite{bou2024solving} cites an older version of this paper. The references to Theorem 6.5, Corollary 6.6, and Proposition 6.8 there should be replaced by Theorem 6.4, Corollary 6.5, and Proposition B.2 here.}
In addition to classifying 2-solvable operators of order~3, we also develop tools needed to study higher order operators; see Sections \ref{section:restriction-induction}--\ref{section:main}.

The above theorem is the difference analog of what M. Singer \cite{singer1985solving} proved for differential operators.
Aspects of \cite{singer1985solving} were subsequently generalized:
M. Singer \cite{singer1988algebraic} gave representation-theoretic criteria for solvability by lower order operators;
A. Person \cite{person2002solving} used \cite{singer1988algebraic} to classify order $4$ operators that are solvable in this way; 
M. van Hoeij \cite{hoeij2007solving} developed an algorithm for \cite{singer1985solving};
K. A. Nguyen \cite{nguyen2009d} and Nguyen-van der Put \cite{nguyen2010solving} generalized and streamlined \cite{singer1988algebraic}, and gave other solvability criteria.

To analyze Cases~(b) and~(c), we introduce representation-theoretic ideas that contextualize the results of Hendriks-Singer \cite{hendricks1999solving}.
We will also formulate results in terms of difference modules to simplify proofs. 
Viewing difference modules as ``twisted'' $\mathbb{Z}$-representations allows us to restrict and induce difference modules along subgroups of $\mathbb{Z}$.
Two natural operations on sequences---sectioning and interlacing---readily correspond to 
restriction and induction.
This was hinted at by the work of
K. Amano \cite[Section 3.7]{amano2006relative} which also generalized \cite{hendricks1999solving} to Artinian simple module algebras. 

As future work, we expect that (from private communication with M. van der Put) an adaptation of \cite{nguyen2009d} and \cite{nguyen2010solving} will determine when difference modules are solvable in terms of lower-dimensional difference modules. This would simplify the proofs given here and describe $d$-solvability in general.

Our paper is organized as follows. 
\Cref{section:background} recalls relevant facts about difference operators and modules.
\Cref{section:eulerian} formally defines $2$-expressible sequences and relates them to Eulerian groups.
This section also characterizes operators whose solutions are $2$-expressible.
The symmetric power construction as recalled in \Cref{section:symmetric-powers} gives us a criterion (\Cref{prop:sym-square}) to check for Case (c) of our main theorem.
\Cref{section:restriction-induction} defines induced and restricted difference modules.
\Cref{section:6} gives examples of how these operations help with simplifying difference operators and motivates the rest of the paper.
\Cref{section:semisimplicity} describes theoretical results on induction and restriction,
most notably \Cref{thm:absolute-irreducibility}, which characterizes the irreducible but not absolutely irreducible modules.
In \Cref{section:main}, we prove our main result, \Cref{thm:general-classification}. 

The paper ends with three appendices. \Cref{appendix:groebner} completes the proof of \Cref{prop:sym-square}. 
\Cref{appendix:mackey-formula} covers a difference analogue of Mackey's formula that go beyond the results discussed in \Cref{section:semisimplicity}.
Finally, \Cref{appendix:absolute irreducibility} explains how our notion of absolute irreducibility aligns with other common usages of the term.\\

\noindent\textbf{Acknowledgements:} 
The first and second authors were supported by NSF grant CCF-2007959. 
We thank Ettore Aldrovandi for helpful conversations and Marius van der Put for helpful comments.
We are grateful to the referees whose feedback helped improve this paper.
\\

\noindent\textbf{Assumptions:} In this paper,
$(F,\phi)$ is a \emph{difference field}, i.e., $F$ is a field and $\phi:F\to F$ is a field automorphism.
The subset $C\coloneqq \left\lbrace f\in F\mid \phi(f)=f\right\rbrace$ is an algebraically closed field of characteristic zero.

\section{Background}
\label{section:background}
In this section, we state some facts about difference operators and difference Galois theory needed for this paper. 
We will start by describing difference operators in terms of difference modules since the latter is a more flexible setting to develop our results.
More details of these backgrounds can be found in \cite[Chapter 1]{put2006galois}.

\subsection{Difference modules}

Formally, a $D$-\emph{module} is a left module over the noncommutative ring $D\coloneqq F[\Phi,\Phi^{-1}]$ with multiplication $\Phi\circ f= \phi(f)\circ \Phi$ for all $f\in F$. 
Unless stated otherwise, all difference modules are finite-dimensional as $F$-vector spaces. 
This dimension is the \emph{dimension} of the difference module.

Let $M$ be a $D$-module.
If $\{m_{i}\}$ is an $F$-basis of $M$ and $\Phi(m_{i})=\sum_{j} a_{ij}m_{j}$, then $M$ corresponds to the matrix difference equation $\phi(Y)=(a_{ij})^{-1}Y$ according to \cite[Section 1.4]{put2006galois}.

An element of $D$ is a \emph{difference operator} over $(F,\phi)$. 
To a nonzero difference operator $L$, we associate the $D$-module
\[
    M_L\coloneqq D/DL.
\]
The \emph{order} of $L$ is defined as the dimension of $M_{L}$.
Two difference operators $L$ and $L'$ are \emph{gauge equivalent} if $M_L$ and $M_{L'}$ are isomorphic $D$-modules.

Let $M$ be a $D$-module. 
A \emph{minimal operator} of $m\in M$ is a 
generator of the left ideal $\{L\in D\mid L(m)=0\}$ of $D$. 
We call $m$ a \emph{cyclic vector} if $D\cdot m=M$, in which case a minimal operator of $m$ has order $\dim_{F}M$.
Therefore, a difference operator is equivalent to a $D$-module with a choice of cyclic vector.
If $F$ has an element $f$ for which $\phi^{m}(f)\neq f$ for all $m\ge 1$, then every $D$-module has a cyclic vector \cite[Theorem B.2]{hendricks1999solving}.

Solutions of difference operators and difference modules reside within difference rings.
Formally, a \emph{difference ring} consists of a ring and an automorphism of the ring. 
Let $(R,\phi)$ be a $D$-\emph{algebra}, i.e., $(R,\phi)$ is a difference ring with a left $D$-action such that $\Phi\cdot r=\phi(r)$.
The \emph{solution space} of an operator $L$ in $R$ is the set $V_{R}(L)\coloneqq \{f\in R\mid L(f)=0\}$. 
The \emph{solution space} of a $D$-module $M$ is the set
\[
    V_{R}(M)\coloneqq \Hom_{D}(M,R).
\]
An $R$-\emph{point} of $M$ is an element of $V_{R}(M)$. 
If $m$ is a cyclic vector of $M$ and $L$ is the\footnote{\label{note1} Minimal operators and LCLM's are unique up to left multiplication by units in $D$.} minimal operator of $m$, then the evaluation map
\[
    V_{R}(M)\to V_{R}(L)\colon\quad s\mapsto s (m)
\]
defines a bijection between the $R$-points of $M$ and the solutions of $L$ in $R$. 
In general, $V_{R}(M)$ is a $C$-vector space of dimension at most $\dim_{F} M$. We say $V_{R}(M)$ is a \emph{full solution space} if its $C$-dimension is $\dim_{F} M$.

Let $L$ and $L'$ be difference operators. 
Then $\LCLM(L,L')$ is the\footnoteref{note1} generator of the left ideal $DL\cap DL'$ of $D$. 
Note that 
\[
    V_{R}(\LCLM(L,L'))=V_{R}(L)+ V_{R}(L')
\]
when these are full solution spaces.

\subsection{Difference Galois theory}
\label{subsection:picard-vessiot}

Let $M$ be a $D$-module with an associated matrix difference equation $\phi (Y)=AY$, $A\in \GL_{n}(F)$. 
A \emph{Picard-Vessiot ring} for $M$ over $F$ is a $D$-algebra $R$ satisfying the following conditions.
\begin{enumerate}
    \item $\left\lbrace f\in R\mid \phi(f)=f\right\rbrace = C$.
    \item The only ideals $I$ of $R$ that satisfy $\phi(I)\subseteq I$ are $0$ and $R$.
    \item $R = F[z_{ij},\det(Z)^{-1}\mid 1\le i, j\le n]$ for some $Z = (z_{ij})\in \GL_{n}(R)$ satisfying $\phi(Z)=AZ$.
\end{enumerate}

A Picard-Vessiot ring for $M$ exists, is unique up to an isomorphism of $D$-algebras \cite[Proposition 1.9]{put2006galois}, and is independent of the chosen matrix difference equation \cite[page 24]{put2006galois}.

Let $R$ be a Picard-Vessiot ring of $M$. The \emph{difference Galois group} of $R/F$ is the group $\Gal(R/F)$ of $D$-algebra automorphisms of $R$. Let $G=\Gal(R/F)$. By \cite[Theorem 1.13]{put2006galois}, $G$ is a linear algebraic group over $C$. For $\sigma\in G$ and $s\in V_{R}(M)$, $(\sigma,s)\mapsto\sigma\circ s$ defines a $G$-action on $V_{R}(M)$ and induces an injective homomorphism $G\to \GL(V_{R}(M))$ of algebraic groups. 

The total ring of fractions $K$ of a Picard-Vessiot ring of $M$ is called a \emph{total Picard-Vessiot ring} of $M$. In place of $M$, we often refer to the (total) Picard-Vessiot ring of an operator $L$ or matrix equation associated to $M$. We also let $\Gal(K/F)$, $\Gal(M/F)$, $\Gal(M)$, and $\Gal(L)$ refer to $\Gal(R/F)$.

\begin{proposition}
[{\cite[Corollary A.7]{hendricks1999solving}}]
\label{prop:solution-space-correspondence}
    Let $M$ be a $D$-module that has a cyclic vector and a total Picard-Vessiot ring $K$. The map $N\mapsto V_K(N)$ defines a bijection between the $D$-module quotients of $M$ and the $\Gal(M)$-invariant subspaces of $V_K(M)$. Moreover, $V_K(N)$ is a full solution space of $N$ for each quotient module $N$ of $M$.
\end{proposition}

The Galois correspondence holds for difference modules.

\begin{proposition}
\label{prop:galois-correspondence}
Let $K/F$ be a total Picard-Vessiot extension with difference Galois group $G$. Let $\mathcal{F}$ be the set of $D$-subalgebras $E$ of $K$ such that every non zero-divisor of $E$ is a unit of $E$. Let $\mathcal{G}$ be the set of algebraic subgroups of $G$. 
    \begin{enumerate}[label=(\alph*)]
        \item\label{galois-a} The maps $\mathcal{F}\to \mathcal{G}\colon E\mapsto \Gal (K/E)$ and $\mathcal{G}\to \mathcal{F}\colon H\mapsto K^{H}$ are inverses. 
        \item\label{galois-b}\sloppy If $E/F$ is a total Picard-Vessiot extension with $E \subseteq K$, then $\Gal(K/E)\trianglelefteq G$.
    \end{enumerate}
\end{proposition}
\begin{proof}
    \ref{galois-a} is \cite[Theorem 1.29]{put2006galois}, and 
    \ref{galois-b} follows from the computation 
    \[
        \sigma \Gal (K/E)\sigma^{-1} = \Gal (K/\sigma (E)) = \Gal (K/E)
    \]
    which holds for all $\sigma\in G$ since $E/F$ is a total Picard-Vessiot extension.
\end{proof}

\section{2-expressible sequences}
\label{section:eulerian}

This section describes the order $2$ generalization of Liouvillian sequences that we will call $2$-expressible sequences. 
Similar to how \cite{hendricks1999solving} classified difference operators with Liouvillian solutions, our goal is to classify operators with 2-expressible solutions, with some partial progress described in \Cref{section:6} and a definitive result in \Cref{thm:general-classification}.
For this, we study the relation between difference operators, $2$-expressible solutions, and the Galois group.
Throughout this section, we assume $(F,\phi)=(C(x),x\mapsto x+1)$.

We first recall the definition of a sequence and some basic constructions on sequences as in \cite[Section 3]{hendricks1999solving}. 
Let $\mathbb{N}=\lbrace 0,1,2,\dots\rbrace$.
We define $\mathcal{S}$ to be the quotient set of functions $s\colon\mathbb{N}\to C$ under the equivalence relation $s\sim t$ if $s(n) = t(n)$ for all but finitely many $n$. 
An element of $\mathcal{S}$ is called a \emph{sequence} and is represented as a function $s\colon\mathbb{N}\to C$ or as a list $(s(0),s(1),s(2),\dots)$. 
The set $\mathcal{S}$ becomes a difference ring under function addition and multiplication, and $\phi (s(n))\coloneqq s(n+1)$.
Since rational functions have finitely many poles, $F=C(x)$ embeds into $(\mathcal{S},\phi)$ by evaluation at $\mathbb{N}$.
Every difference operator over $F$ has a full solution space in $\mathcal{S}$ by \cite[Theorem 8.2.1]{petkovvsek1997wilf}.

We construct new sequences from old ones as follows. For $s\in \mathcal{S}$,
\[
    \Sigma s\coloneqq \lbrace t\in \mathcal{S}\mid (\phi-1)t=s\rbrace
\]
is the set of \emph{indefinite sums} of $s$. 
The \emph{interlacing} of $s_{0},s_{1},\dots,s_{m-1}\in \mathcal{S}$ is the sequence $t$ defined by $t(mn+i)=s_{i}(n)$ for $n\ge 0$ and $0\le i<m$. 
The $m$-\emph{interlacing} of $s$ \emph{with zeros} is defined as the interlacing of $s$ with $(m-1)$-many zero sequences. 

\begin{remark}\hfill
    \begin{enumerate}
        \item The interlacing of $s_{0},s_{1},\dots,s_{m-1}$ is the sum 
        $\sum \phi^{-i}(t_i)$
        where $t_{i}$ is the $m$-interlacing of $s_{i}$ with zeros. 
        \item If $s$ satisfies an operator $L(x,\Phi)=\sum_{i}f_{i}(x)\Phi^{i}$, then the $m$-interlacing of $s$ with zeros satisfies the operator $L(x/m,\Phi^{m})=\sum_{i}f_{i}(x/m)\Phi^{im}$.
    \end{enumerate}
\end{remark}

In \cite[Definition 3.3]{hendricks1999solving} the ring of \emph{Liouvillian sequences} $\mathcal{L}$ is defined as the smallest $D$-subalgebra
$R$ of $\mathcal{S}$ that satisfies the following conditions.
\begin{enumerate}
    \item\label{liouvillian-1} Any order $1$ operator in $D$ has a full solution space in $R$. 
    \item\label{liouvillian-2} If $s\in R$, then $\Sigma s\subseteq R$.
    \item\label{liouvillian-3} If $s\in R$, then the $m$-interlacing of $s$ with zeros is in $R$ for all $m\ge 1$.
\end{enumerate}

Condition \eqref{liouvillian-3} and the closure of $\mathcal{L}$ under difference ring operations imply that $\mathcal{L}$ is closed under arbitrary interlacing.

We can now define $2$-expressible sequences in a similar way to the Eulerian solutions considered in \cite{singer1985solving}.

\begin{definition}
\label{def:eulerian}
    Consider the family of $D$-subalgebras $R$ of $\mathcal{S}$ that satisfy conditions \eqref{liouvillian-2} and \eqref{liouvillian-3} above, as well as
    \begin{enumerate}[resume]
        \item\label{liouvillian-4} Any order $2$ operator in $D$ has a full solution space in $R$. 
    \end{enumerate}
    We define $\mathcal{E}$ to be the smallest $D$-algebra within this family.
    An element of $\mathcal{E}$ is called a \emph{$2$-expressible sequence}.
    An $\mathcal{S}$-point $s\colon M\to \mathcal{S}$ of a $D$-module $M$ is \emph{$2$-expressible} if its image is contained in $\mathcal{E}$.
\end{definition}

Clearly a solution of an operator $L$ is $2$-expressible if and only if the corresponding $\mathcal{S}$-point of $M_{L}$ is $2$-expressible.

\cite[Theorem 3.4]{hendricks1999solving} shows how Liouvillian sequences correspond to solvable Galois groups.
We will give an analog for $2$-expressible sequences and Eulerian Galois groups.
\begin{definition}\cite[page 667]{singer1985solving} An algebraic group $G$ is said to be \emph{Eulerian} if there exist subgroups $1=G_{0}\subset \dots \subset G_{n}=G$ such that
$G_{i}\triangleleft G_{i+1}$ and $G_{i+1}/G_{i}$ is finite or isomorphic to $\mathbb{G}_{a}$, $\mathbb{G}_{m}$, or $\PSL_{2}(C)$.\end{definition} 
Eulerian groups are closed under taking closed subgroups, quotients, and direct products \cite[Lemma 2.2]{singer1985solving}. 
Since the connected closed subgroups of $\SL_{2}(C)$ are either solvable or $\SL_{2}(C)$ by \cite[page 31]{kaplansky1957introduction} (see also \cite[page 7]{kovacic1986algorithm} and \cite[page 28]{put2003galois}), 
the Galois group of an order $2$ operator is Eulerian.

We also need to consider operators corresponding to 2-expressible sequences. 

\begin{definition}\label{def:fully-2-solvable-operator}
    We define $\mathcal{F}$ to be the smallest subset of $D$ that satisfies the following conditions.
    \begin{enumerate}[label=(\roman*)]
        \item\label{item:operation-1} Any order $2$ operator in $D$ is in $\mathcal{F}$.
        \item\label{item:operation-2} If $L,L'\in \mathcal{F}$, then $L \circ (\Phi - 1)$, $\Phi^{i}\circ L \circ \Phi^{-i}$, $\LCLM(L,L')$, and $L\cs L'$ are in $\mathcal{F}$.
        \item\label{item:operation-6} If $L$ is in $\mathcal{F}$, then every right factor of $L$ is in $\mathcal{F}$.
        \item\label{item:operation-7} If $L = L(x, \Phi)$ is in $\mathcal{F}$, then so is $L(x/m, \Phi^m)$ for any $m \ge 1$.
    \end{enumerate}
    The minimal number of times the operations \ref{item:operation-2}--\ref{item:operation-7} 
    needed
    to obtain $L\in \mathcal{F}$ is called the \emph{complexity} of $L$ (the complexity is 0 if $L$ has order 2).
\end{definition}

\begin{theorem}
    \label{thm:eulerian-group}
    Let $s\in \mathcal{E}$. A minimal operator of $s$ exists, is in $\mathcal{F}$, and has an Eulerian Galois group.
\end{theorem}
\begin{proof}

Let $\mathcal{V_{D}}=\{s \in \mathcal{S} \mid \exists \, L \in D - \{0\},\,L(s)=0\}$, the set of elements of $S$ that satisfy a recurrence relation, and let
     \begin{equation*}
        \begin{aligned}
            A &=\{s\in \mathcal{V_{D}}\mid \text{minimal operator of }s\text{ has Eulerian Galois group}\}\\
            B&=\{s\in \mathcal{V_{D}}\mid \text{minimal operator of }s\text{ is in } \mathcal{F}\}.
        \end{aligned}
    \end{equation*}
    Solutions of order $2$ operators are clearly in $B$, and in $A$ by the discussion preceding \Cref{def:fully-2-solvable-operator}.
    To prove the theorem, that $A$ and $B$ contain $\mathcal{E}$, it remains to show that $A$ and $B$ are $D$-algebras closed under \eqref{liouvillian-2} and \eqref{liouvillian-3}.

    Let $s,t\in A$ be solutions to operators with Eulerian Galois groups $G,H$.
    The proof of \cite[Theorem 3.4]{hendricks1999solving} shows that $\phi^i (s)$, $\Sigma s$, and the $m$-interlacing of $s$ with zeros are solutions to operators whose Galois groups are $G$, a subgroup of the semidirect product of $G$ and $\mathbb{G}_{a}^{n}$ for some $n$, or a finite cyclic extension of $G^{m}$.
    Moreover, $st$ and $s-t$ lie in a composite Picard-Vessiot ring with Galois group $G'\le G\times H$.
    By \cite[Lemma A.8]{hendricks1999solving}, both $st$ and $s-t$ are solutions to operators with Galois groups that are quotients of $G'$.
    Eulerian groups are closed under taking closed subgroups, quotients, and direct products \cite[Lemma 2.2]{singer1985solving}, 
    so $\phi(s), \Sigma s$, the $m$-interlacing of $s$ with zeros, $st$, and $s-t$ satisfy operators with Eulerian Galois groups and so are in $A$. Thus $\mathcal{E} \subseteq A$.

    Let $s,t\in B$ have minimal operators $L,L'$ in $\mathcal{F}$. 
    Then $s+t$, $st$, $\phi^{i}(s)$, and the $m$-interlacing of $s$ with zeros satisfy $\LCLM(L,L')$, $L\cs L'$, $\Phi^{i}\circ L\circ \Phi^{-i}$, and $L(x/m,\Phi^{m})$, respectively, which are all in $\mathcal{F}$. 
    By \ref{item:operation-6}, their minimal operators are also in $\mathcal{F}$.
\end{proof}

\begin{definition} \label{solv2}
    An operator is \emph{$2$-solvable} if it has a nonzero 2-expressible solution. It is \emph{fully $2$-solvable} if it has a basis of 2-expressible solutions.
\end{definition}

\Cref{thm:fully-2-solvable} below shows that any minimal operator of $s\in \mathcal{E}$ is fully $2$-solvable. First we need a lemma.

\begin{lemma}\label{lemma:shift}
    If $L = L(x, \Phi)$ is in $\mathcal{F}$ then so is $L(x+c, \Phi)$ for any $c\in C$. This operation preserves the complexity of $L$.
\end{lemma}
\begin{proof}
    Induct on the complexity of $L$, 
    noting that $\sigma_{c}\colon x \mapsto x+c$ commutes with \ref{item:operation-1}--\ref{item:operation-6} and almost-commutes with \ref{item:operation-7}, 
    i.e., if $S(L)\coloneqq L(x/m,\Phi^{m})$,
    then $\sigma_{c}(S(L))=S(\sigma_{c/m}(L))$.
\end{proof}

\begin{theorem}\label{thm:fully-2-solvable}
     If $L$ is in $\mathcal{F}$ then $L$ is fully $2$-solvable. In particular, any minimal operator of a $2$-expressible sequence is fully $2$-solvable.
\end{theorem}
\begin{proof}
    The first claim and \Cref{thm:eulerian-group} imply the second claim.
    To prove the first claim, induct on the complexity of $L$. The base case is clear. 
    For the inductive step, suppose that $L,L'\in \mathcal{F}$ satisfy $V_{\mathcal{S}}(L),V_{\mathcal{S}}(L')\subseteq \mathcal{E}$.
    By \Cref{lemma:shift}, $L(x+i/m,\Phi)\in \mathcal{F}$. By the inductive hypothesis, $V_{\mathcal{S}}(L(x+i/m,\Phi))\subseteq \mathcal{E}$.
    Let $L''$ be a right factor of $L$. 
    Then the induction step for \ref{item:operation-2}--\ref{item:operation-6} is verified by the following computations.
    \begin{equation*}
        \begin{aligned}
            &V_{\mathcal{S}}(L\circ (\Phi-1))=\bigcup_{s\in V_{\mathcal{S}}(L)}\Sigma s \quad\subseteq \mathcal{E}\\
            &V_{\mathcal{S}}(\Phi^{i}\circ L\circ \Phi^{-i})=\phi^{i}(V_{\mathcal{S}}(L))\subseteq \mathcal{E}\\
            &V_{\mathcal{S}}(\LCLM(L,L'))=V_{\mathcal{S}}(L)+V_{\mathcal{S}}(L')\subseteq \mathcal{E}\\
            &V_{\mathcal{S}}(L\cs L') \subseteq \mathcal{E}\\
            &V_{\mathcal{S}}(L'')\subseteq V_{\mathcal{S}}(L)\subseteq \mathcal{E}\\
        \end{aligned}
    \end{equation*}
    Here $V_{\mathcal{S}}(L\cs L')\subseteq \mathcal{E}$ uses equation \eqref{eq:sym-prod} in \Cref{section:symmetric-powers}.
    Finally $V_{\mathcal{S}}(L(x/m,\Phi^{m}))$ consists of interlacing sequences in $V_{\mathcal{S}}(L(x+i/m,\Phi))$ for $0\le i<m$, so it is also in $\mathcal{E}$.
\end{proof}

A question remains: is every operator with an Eulerian Galois group $2$-solvable? For order $3$ operators, this follows from the proof in \Cref{section:main}.

\section{Symmetric powers}
\label{section:symmetric-powers}

Let $M$ and $N$ be $D$-modules. The tensor product $M\otimes N$ over $F$ is a $D$-module via $\Phi (m\otimes n)\coloneqq \Phi (m)\otimes \Phi (n)$. The tensor power $M^{\otimes d}$ is naturally a $D$-module, as is the subset $S^{d}(M)\subseteq M^{\otimes d}$ of elements fixed by permuting the tensor factors.
We call $S^{d}(M)$ the $d$th \emph{symmetric power} of $M$. 

The \emph{symmetric product} $L\cs L'$ of operators $L$ and $L'$ is a minimal operator of $1\otimes 1$ in $M_{L}\otimes M_{L'}$. Like \cite[Corollary 2.19]{put2006galois}, we have
\begin{equation}\label{eq:sym-prod}
V_{\mathcal{S}}(L\cs L')=\spann_{C}\left\lbrace u\cdot v\,\middle\vert\, u\in V_{\mathcal{S}}(L), v\in V_{\mathcal{S}}(L')\right\rbrace.
\end{equation}
We let $L^{\cs 2}= L\cs L$.

We now state an analog of a result by Fano, Singer, van Hoeij--van der Put \cite[Proposition 3.1]{hoeij2006descent}, and Nguyen--van der Put \cite[Corollary 2.2]{nguyen2010solving} that determines when a $D$-module of dimension $3$ is an almost symmetric square.

\begin{proposition}[Symmetric Square Criterion]
\label{prop:sym-square}
  Let $F$ be a $C_{1}$-field. Let $M$ be an irreducible $D$-module of dimension $3$. Then $S^{2} (M)$ has a $D$-submodule of dimension $1$ if and only if $M\cong S^{2} (N)\otimes P$ for some $D$-modules $N$ and $P$ of dimensions $2$ and $1$, respectively.
\end{proposition}
\begin{proof}
$\Leftarrow$: The canonical $D$-module epimorphism $S^2(S^2(N))\twoheadrightarrow S^4(N)$ has a kernel of dimension $1$.
Hence $S^2(M) \cong S^2(S^2(N)) \otimes S^2(P)$ has a submodule of dimension $1$ as well.

$\Rightarrow$: Let $q\in S^2(M)$ span a $D$-submodule of dimension $1$, i.e. $\Phi (q) = rq$ for some $r\in F^{*}$. We will show that $q$ is irreducible (as quadratic polynomial in an $F$-basis $x_1, x_2, x_3$ of $M$) which is equivalent to saying that the associated bilinear form is nondegenerate. One way to see that $q$ is irreducible is because otherwise, its factors would generate a nontrivial submodule of $M$, but $M$ is irreducible.

(Alternatively,
identify $q$ with its associated symmetric bilinear form on $M^{*}$. Recall that $M^*\coloneqq \Hom_F(M,F)$ is a $D$-module under the conjugation action $\Phi(l)\coloneqq \Phi_F\circ l\circ \Phi_M^{-1}$ and enjoys properties analogous to those of its differential counterpart in \cite[page 45]{put2003galois}.
The computation
\[
0=\Phi(q(x,\Phi^{-1}(y)))=\Phi(q)(\Phi(x),y)=r\cdot q(\Phi(x),y)\quad\text{for all $y\in M^*$}
\]
verifies that $Q\coloneqq\{x\in M^{*}\mid q(x,M^{*})=0\}$ is a $D$-submodule of $M^*$. Also $q$ is nonzero so $Q$ cannot be all of $M^*$.
Since $M^*$ is irreducible, $Q=0$. 
Therefore $q$ is nondegenerate.) 

This implies that the quadratic linear form can be reduced to $x_2^2-x_1x_3$. This is well known, but we also give an explicit algorithm in \cite{algo}, where a sequence of elementary linear transformations rewrites $q$ to an $F$-multiple of $x_2^2 - x_1 x_3$ in some $F$-basis $x_1,x_2,x_3$ of $M$.
Alternatively, Witt's decomposition theorem \cite[page 12]{lam2005introduction} makes $q$ an $F$-multiple of the sum of an anisotropic form $x_{2}^{2}$ and a hyperbolic form $-x_1 x_3$. 

Now that we reduced $q$ to $x_2^2 - x_1 x_3$, we will use this form to give explicit equations constraining the structure constants of the difference action on $M$. Write $\Phi (x_{i}) = \sum_j g_{ij}x_{j}$ with $(g_{ij}) \in \GL_3(F)$. 
Since $q$ spans a $1$-dimensional submodule, we have $\Phi(x_2^2 - x_1 x_3) = r ( x_2^2 - x_1 x_3)$ for some $r \in F^*$.
Matching $x_i x_j$-coefficients gives the following relations.
\begin{equation}
\begin{aligned} 
\label{eq:g-relations}
[x_{1}^{2}]&:& 0&= g_{21}^{2}-g_{11}g_{31}\\
[x_{3}^{2}]&:& 0&= g_{23}^{2}-g_{13}g_{33}\\
[x_{1}x_{2}]&:& 0&= 2g_{21}g_{22}-g_{11}g_{32}-g_{31}g_{12}\\
[x_{2}x_{3}]&:& 0&= 2g_{22}g_{23}-g_{12}g_{33}-g_{32}g_{13}\\
[x_{2}^{2}] + [x_{1}x_{3}]&:& 0&= (g_{22}^{2}-g_{12}g_{32}) + (2g_{21}g_{23}-g_{11}g_{33}-g_{31}g_{13})\\
\end{aligned}
\end{equation}
Here, the sum $[x_{2}^{2}] + [x_{1}x_{3}]$ is used to eliminate $r$.

We now decompose $M$ by formally defining vector spaces $N$ and $P$ and giving them difference structures so that $S^{2}(N)\otimes P\to M$ is a $D$-module isomorphism. Let $N$ and $P$ be $F$-vector spaces with respective bases $\{n_{1},n_{2}\}$ and $\{p\}$. The $F$-linear map $S^{2}(N)\otimes P\to M$ taking $n_{i}n_{j}\otimes p$ to $m_{ij}$ will be a $D$-module isomorphism if we can find $a_{ij},b\in F$ defining actions $\Phi (n_{i}) = \sum a_{ij}n_{j}$ on $N$ and $\Phi (p) = b\cdot p$ on $P$ that satisfy
\begin{equation}
\label{eq:g-a-relations}
    \begin{bmatrix}
        g_{11}& g_{12} & g_{13}\\
        g_{21}& g_{22} & g_{23}\\
        g_{31}& g_{32} & g_{33}\\
    \end{bmatrix}
    =b
    \begin{bmatrix}
        a_{11}^{2}& 2a_{11}a_{12} & a_{12}^{2}\\
        a_{11}a_{21}& a_{11}a_{22}+a_{12}a_{21} & a_{12}a_{22}\\
        a_{21}^{2}& 2a_{21}a_{22} & a_{22}^{2}\\
    \end{bmatrix}.
\end{equation}

We now show that the matrix $(g_{ij})$ is diagonal, skew-diagonal, or is not of the forms $A,B,C,D$, where
\[
    A=\begin{bmatrix}
    0 & * & 0\\
    * & * & *\\
    0 & * & 0\\
    \end{bmatrix},\;
    B=\begin{bmatrix}
    0 & * & *\\
    * & * & 0\\
    0 & * & *\\
    \end{bmatrix},\;
    C=\begin{bmatrix}
    * & * & 0\\
    0 & * & *\\
    * & * & 0\\
    \end{bmatrix},\;
    D=\begin{bmatrix}
    * & * & *\\
    0 & * & 0\\
    * & * & *\\
    \end{bmatrix}.
\]
To see this, suppose $(g_{ij})$ is of the form $A$, $B$, $C$, or $D$. Equations $[x_1^2]$ and $[x_3^2]$ of \eqref{eq:g-relations} now force $A$, $B$, and $C$ to have a zero column, contradicting the fact that $(g_{ij})$ is invertible. Both \eqref{eq:g-relations} and $\det(g_{ij})\neq 0$ force $D$ to be diagonal or skew-diagonal. Therefore, $(g_{ij})$ is diagonal or skew-diagonal, or else one of $g_{11}g_{21}\neq 0$, $g_{21}g_{31}\neq 0$, $g_{13}g_{23}\neq 0$, or $g_{23}g_{33}\neq 0$ holds.

Since vertical and horizontal reflections of the square matrices $(a_{ij})$ and $(g_{ij})$ preserve the form of \eqref{eq:g-relations} and \eqref{eq:g-a-relations}, 
we may further assume $(g_{ij})$ is diagonal or $g_{11}g_{21}\neq 0$. In the former case, 
set $b=g_{11}$, $a_{11}=1$, $a_{22}=g_{22}/g_{11}$, and $a_{12}=a_{21}=0$. In the latter case, set
$b= g_{11}$, 
$a_{11}= 1$, $a_{12}= g_{12}/2g_{11}$, 
$a_{21}= g_{21}/g_{11}$, and
$a_{22}= g_{32}/2g_{21}$; for this choice of $a_{ij}, b$,
the relations \eqref{eq:g-a-relations} follow from \eqref{eq:g-relations} by a Gröbner basis computation in \Cref{appendix:groebner}.
\end{proof}

\section{Restriction and induction}
\label{section:restriction-induction}
We now define representation-theoretic operations on difference modules.

\begin{definition}\label{def51} Let $s,t$ be positive integers with $s$ dividing $t$. Consider the noncommutative subrings $D_s=F[\Phi^{s},\Phi^{-s}]$ and $D_t=F[\Phi^{t},\Phi^{-t}]$ of $D$.
\begin{enumerate}
    \item The \emph{restriction} of a $D_s$-module $M$ to $D_{t}$ is defined as the $D_{t}$-module $M$ and denoted by $M\pda^{s}_{t}$.
    \item The \emph{induction} of a $D_{t}$-module $N$ to $D_s$ is defined as the $D_s$-module $D_s\otimes_{D_t}N$ and denoted by $N\pua^{s}_{t}$.
    \item The \emph{restriction} $L\pda^{1}_{t}$ of an operator $L$ is a minimal operator of $1 \in M\pda^{1}_{t}$. $L\pda^{1}_{t}$
    generates the left ideal $D_{t}\cap DL \subseteq D_{t}$.
\end{enumerate}
\end{definition}

\begin{remark}[Basic properties]
\label{remark:basic-properties}
    Let $r | s$ and $s | t$. It follows from definition that $(M\pda^r_s)\pda^s_t=M\pda^r_t$ and that 
    \begin{equation}
    \label{eq:basic:2}
    (N\pua_t^s)\pua_s^r
    =D_r\otimes_{D_s}\left(D_{s}\otimes_{D_t}N\right)
    = D_r\otimes_{D_t}N
    =N\pua^r_t.
    \end{equation}
    Since $\{1, \Phi^{r},\Phi^{2r},\dots,\Phi^{s-r}\}$ is a basis of $D_r$ as a left $D_s$-module and since $\Phi^{ir}\otimes_{D_s} N$ is $D_s$-stable within $D_r\otimes_{D_s} N$, we additionally have
    \begin{equation}
    \label{eq:basic:3}
        N\pua_r^s\pda^r_s\cong (1\otimes N)\oplus(\Phi^{r}\otimes N)\oplus(\Phi^{2r}\otimes N)\oplus\cdots\oplus(\Phi^{s-r}\otimes N).
    \end{equation}
    Combining \eqref{eq:basic:2} and \eqref{eq:basic:3}, we further get
    $$N\pua_t^r\pda^r_s = N\pua^{s}_{t}\pua^{r}_{s}\pda^{r}_{s} = \bigoplus_{i=0}^{s/r-1} (\Phi^{ir}\otimes N\pua^{s}_{t}).
    $$
    For a generalization where $r$ need not divide $s$, see \Cref{thm:mackey-formula}.
\end{remark}

\begin{remark} \label{remark:basic-logic}
It suffices to focus on $M\pda^1_t$ and $N\pua^1_{t}$ because results are easy to generalize to $M\pda^{s}_{t}$ or $N\pua^{s}_{t}$.  After all, to generalize a result from $D_1 = F[\Phi,\Phi^{-1}]$ to $D_s=F[\Phi^s,\Phi^{-s}]$, all we have to do is replace an automorphism $\Phi$ on $F$ with another automorphism $\Phi^s$.
\end{remark}

Restriction and induction have natural interpretations in terms of difference operators and matrix difference equations.

Let $M$ be a $D$-module. If $\phi (Y)=AY$ is the associated matrix equation of $M$ in a given basis $\mathcal{B}$, then $\phi^{t}(Y)=A_{t}Y$ is the associated matrix equation of $M\pda^{1}_{t}$ in the same basis $\mathcal{B}$, where
\[
    A_{t}\coloneqq\phi^{t-1}(A)\cdots\phi^{2}(A)\phi(A)A.
\]
If $m$ is a cyclic vector of $M$, then $m\in M\pda^{1}_{t}$ need not be a cyclic vector, and so the operator $L\pda^{1}_{t}$ only corresponds to a $D_{t}$-submodule of $M\pda^{1}_{t}$.

Now let $N$ be a $D_{t}$-module. If $\phi^{t} (Y)=BY$ is the associated matrix equation of $N$ in basis $\mathcal{C}$, then $\bigcup_{i=0}^{t-1}\Phi^{i}\otimes \mathcal{C}$ is a basis of $N\pua^{1}_{t}$. In this basis, the associated matrix equation is
\[
    \phi(Y)=\tilde{B}Y,\qquad \tilde{B}\coloneqq
    \begin{bmatrix}
        0 & I_{(t-1)d}\\
        B & 0
    \end{bmatrix},
\]
where $d=\dim_{F} N$.

A cyclic vector $n$ of $N$ is also a cyclic vector of $N\pua^{1}_{t}$. So if $R$ is the minimal operator of $n\in N$ (so in particular $R\in D_t$), then $R$ is also the minimal operator of $n\in N\pua^1_t$, except now we view $R$ as an element of $D$.

The table below summarizes the relationship we just discussed about difference modules, matrix difference equations, and difference operators.

\begin{table}[H]
\centering
\renewcommand\tabcolsep{1em}
\renewcommand\arraystretch{1.5}
\begin{tabular}[t]{lll}
    \toprule
    Module & Matrix Equation & Operator \\ 
    \midrule
    $M$ & $\phi(Y)=AY$ & $L\in D$ \\ 
    $M\pda^{1}_{t}$ & $\phi^{t}(Y)=A_{t}Y$ & $L\pda^{1}_{t}\in D_{t}$ \text{if $m\in M\pda^{1}_{t}$ is cyclic} \\
    $N$ & $\phi^{t}(Y)=BY$ & $R\in D_{t}$ \\
    $N\pua^{1}_{t}$ & $\phi(Y)=\tilde{B}Y$ & $R\in D$ \\
    \bottomrule   
\end{tabular}
\end{table}

The formulas $L\pda^{1}_{t}$, $\phi^{t}(Y)=A_{t}Y$, and a variant of $\phi(Y)=\tilde{B}Y$ already appear in, say, \cite[Lemma 5.3]{hendricks1999solving} as $P(\phi^{m})$, \cite[page 20]{put2006galois}, and \cite[page 245]{hendricks1999solving}.
What this means is that induction and restriction have been in use for a long time in various guises. \Cref{prop:facts} translates some of these facts into the language of difference modules for later use. But first, we make an important definition.
\begin{definition}
\label{def:absolute-irreducibility}
    A $D_s$-module $M$ is \emph{absolutely irreducible} if $M\pda^s_t$ is irreducible for all multiples $t$ of $s$. An operator $L$ is \emph{absolutely irreducible} if $M = D/DL$ is absolutely irreducible.
\end{definition}

For difference and differential operators, $L$
is absolutely irreducible when its solution space is 
an irreducible $G^\circ$-module, where $G^\circ$ is the connected component of the identity of the Galois group. In the differential case this occurs when $L$ is irreducible over algebraic extensions
of $\mathbb{C}(x)$. In the difference case, such extensions are replaced with $L\pda^1_t$ (See Appendix C).

\begin{proposition}
    \label{prop:facts} Let $(F,\phi)=(C(x),x\mapsto x+1)$ and $M$ be a $D$-module.
    \begin{enumerate}
        \item \label{fact-1} If $\dim_{F} M>1$ and $\Gal(M)$ is solvable, then $M\pda^{1}_{t}$ is reducible for some $t\ge 1$ (i.e., $M$ is not absolutely irreducible).
        \item \label{fact-2} For some $t\ge 1$, the total Picard-Vessiot ring of $M\pda^{1}_{t}$ is a difference field $(K,\phi^{t})$ and $\Gal(M\pda^{1}_{t})=\Gal(M)^{\circ}$.
    \end{enumerate}
\end{proposition}
\begin{proof}
    \eqref{fact-1}:
    We are done if $M$ is reducible. Suppose $M$ is irreducible with an associated operator $L$.
    Since $\Gal (L)$ is solvable, some nonzero solution of $L$ is a Liouvillian sequence \cite[Theorem 3.4]{hendricks1999solving}. 
    By \cite[Theorem 5.1]{hendricks1999solving}, some nonzero solution of $L$ is an interlacing of hypergeometric sequences. 
    By \cite[Corollary 4.3]{hendricks1999solving}, 
    $L$ is gauge equivalent to some $\Phi^{t}+f$ for some $f\in F$. 
    Thus $D_t/D_t(\Phi^t+f)$ is a $1$-dimensional quotient of $(D/DL)\pda^1_t\cong M\pda^1_t$. Because $\dim_FM>1$, we have just found a nontrivial quotient of $M\pda^1_t$, making $M\pda^1_t$ reducible.

    \eqref{fact-2}: 
    The Picard-Vessiot ring $S$ of $M$ has the form $\bigoplus_{i=0}^{t-1}\phi^{i}(R)$ for some difference ring $(R,\phi^{t})$ where $R$ is an integral domain.
    By \cite[Lemma 1.26]{put2006galois}, $M\pda_{t}^{1}$ has the Picard-Vessiot ring $(R,\phi^{t})$ so we set $K= \Frac(R)$.

    By \cite[Proposition 1.20]{put2006galois}, $S$ (resp. $R$) is isomorphic to the coordinate ring of the algebraic group representing $\Gal(M)$ (resp. $\Gal(M\pda_{t}^{1})$) after base change to $F$. Since $R$ is a connected component of $S$, we have $\Gal(M\pda_{t}^{1})=\Gal (M)^{\circ}$.
\end{proof}

\subsection{Examples of absolute (ir)reducibility} \label{section:6}

Let $M = D/DL$. The papers \cite{ singer1985solving, hessinger2001computing,person2002solving} study when a {\em differential} equation is 2-solvable. For order 4, one of the cases to check is if $M \cong M_1 \otimes M_2$ for some $M_1$ and $M_2$ with $\dim(M_1)=\dim(M_2)=2$. Let's call this the {\em product-case}. One can detect the product-case by factoring the exterior square: check if $\bigwedge^2(M)\cong N_1 \oplus N_2$ for some $N_1$ and $N_2$ of dimension $3$ and then write $N_1$ and $N_2$ as symmetric squares of modules of dimension 2.

To develop such results for the difference case, there is subtlety that needs to be studied, which we illustrate with an example.


\begin{example} \label{A227845}

In the OEIS database, the sequence A227845 has an order $4$ minimal operator $$L=(x+4)^2\Phi^4-2(3x^2+21x+37)\Phi^3+2(3x^2+15x+19)\Phi-(x+2)^2.$$

Now $M$ and $\bigwedge^2(M)$ are irreducible, but $\bigwedge^2(M \pda^{1}_{2})$ is reducible.
So this example falls in the product-case, but only after restriction (Definition~\ref{def51}).
We managed to find the following product:
\begin{equation} \label{prodU}
    A227845(n) = U(n-1)U(n)
\end{equation}
where $U(-1) = U(0)=1, U(1)=2, U(2)=7/2,$ and 
\begin{equation}n^2U(n) = 2(3n^2-3n+1)U(n-2) - (n-1)^2U(n-4) \label{ord4rec} \end{equation}

Does \eqref{prodU} write A227845 as a product of solutions of lower order recurrences? The answer is: it depends.

Strictly speaking, \eqref{ord4rec} is a recurrence for $U(n)$ of order 4. However, it is also a recurrence for $v_0(n) := U(2n)$ of order 2. Likewise for $v_1(n) := U(2n-1)$.  So if we {\em restrict} A227845 to even-numbered terms then equation~\eqref{prodU} indeed writes $A227845(2n)$ as a product $v_1(n) v_0(n)$ whose factors satisfy second order recurrences. And likewise for $A227845(2n+1) = v_0(n)v_1(n+1)$.  So A227845 is an interlacing of two sequences, each of which is a product of solutions of second order equations.

In summary: A227845 is not a product of lower-order sequences, but it becomes one if we restrict it to even-numbered terms, or to odd-numbered terms.
Likewise $\bigwedge^2(M)$ is an irreducible $D$-module, but becomes reducible when restricted to a $D_2$-module.
\end{example}

\begin{example} In the OEIS database, both the sequences A105151 and A105216 satisfy an order $4$ operator $$L=\Phi^4-(x+4)\Phi^3-(x+3)\Phi-1.$$ 

Now $\bigwedge^2(M)$ is again irreducible, but $\bigwedge^2(M \pda_{2}^{1}) \cong N_1 \oplus N_2$ for some $N_1$ and $N_2$ of order $3$. We used \cite{bou2024solving} to write $N_1$ and $N_2$ as symmetric squares, reducing the orders to two. We then found solutions in terms of Bessel functions using \cite{cha2010solving}. Putting everything together we found: \begin{equation} \label{eq1} A105151(n) = \round\Bigg(c\BesselK\Big(\frac{n}{2}+\frac{1}{2},1\Big)\BesselK\Big(\frac{n}{2}+1,1\Big)\Bigg)\end{equation}
%
%
where
   $$c = \sqrt{\frac{2}{\pi}}\Bigg(\cosh(1)\BesselI(0,1) + \sinh(1)\BesselI(1,1)\Bigg).$$
The formula for A105216($n$) is the same as~(\ref{eq1}) except this time:
	$$c = \sqrt{\frac{2}{\pi}}\Bigg(\sinh(1)\BesselI(0,1)+\exp(-1)\BesselI(1,1)\Bigg).$$ 

\end{example}

\subsubsection{Motivation of the Paper}

Can one always decide if a recurrence of order 3 or 4 is 2-solvable or not?
This leads to the following questions:

\begin{enumerate} 
\item Which cases can occur? 

\item How to solve those cases? 

\item How to {\em prove} non-2-solvability for recurrences that we do not solve?\end{enumerate}

\begin{remark}[Remarks on questions (1), (2), and (3)]\hfill

\begin{enumerate}
\item Let $N=\bigwedge^2(M)$ with $M$ as in \Cref{A227845}. Then $N$ is irreducible and $N \pda_{2}^{1}$ is reducible. This raises the following question.

Suppose that $N$ is not absolutely irreducible, in other words, there exists $p$ for which $N \pda_{p}^{1}$ is reducible. Which values of $p$ would need to be checked? This motivates \Cref{section:semisimplicity}.

\item For order $3$, see \cite{bou2024solving}. Order $4$ is work in progress; we need to develop an algorithm for each potential case, one of which was illustrated in \Cref{A227845}.  

\item This is the most difficult question. The only way to answer it in the differential case was through Galois theory. The main goal in this paper is to tackle this problem for the difference case. We complete this task for order~3 in \Cref{section:main} while \Cref{section:semisimplicity} gives theory that, as illustrated in the examples, will be needed for order~4.
\end{enumerate}
The three tasks are quite different in nature.  If we limit ourselves to say orders 3 and 4, then task (1), to plausibly list all possible cases, is the easiest task.
%
Task (2) is to design an algorithm for each case.
The main goal in this paper is to develop theory needed for task~(3).
\end{remark}

\section{Semisimplicity and absolute irreducibility}
\label{section:semisimplicity}
We now develop some representation theory in the context of difference modules to decompose irreducible modules into absolutely irreducible modules.

As with differential modules \cite[Exercise 2.38 (4)]{put2003galois}, the tensor product of two semisimple $D$-modules is again semisimple.
Therefore the class of semisimple $D$-modules is closed under linear algebra operations like taking submodules, quotients, tensor products, symmetric powers, and direct sums. 
We now show that restriction and induction preserve semisimplicity as well. 
See \cite[Chapter 17]{lang2012algebra} for general facts on semisimplicity.

\begin{proposition}\label{prop:semisimple}
    Let $M$ be an irreducible $D_s$-module and $N$ an irreducible $D_t$-module, with $s|t$. Then $M\pda^{s}_{t}$ and $N\pua^{s}_{t}$ are semisimple.
\end{proposition}
\begin{proof}
(It would be sufficient to only consider $s=1$ by \Cref{remark:basic-logic}.)
    To show $M\pda^{s}_{t}$ is semisimple, let $P$ be an irreducible $D_{t}$-submodule.
    Since $M$ is irreducible, the $D_s$-submodule $P+\Phi^s(P)+\Phi^{2s}(P)+\dots+\Phi^{t-s} (P)$ equals $M$. But this writes $M\pda^{s}_{t}$ as a sum of irreducible $D_t$-submodules $\Phi^{is}(P)$.

    We next show $N\pua^{s}_{t}$ is semisimple. Let $Q$ be a $D_s$-submodule of $N\pua^{s}_{t}$. Since $N\pua^{s}_{t}\pda^{s}_{t} = \bigoplus_{i}\Phi^{is}\otimes N$ is semisimple, there exists a projection
    $\alpha\colon N\pua^{s}_{t}\pda^{s}_{t}\twoheadrightarrow Q\pda^{s}_{t}$
    of $D_{t}$-modules.
    The map
    \[
    \tilde{\alpha}\colon N\pua^{s}_{t} \twoheadrightarrow N\pua^{s}_{t}\colon \quad
    n\mapsto \frac{s}{t}\sum_{i=0}^{\frac{t}{s}-1}\left(\Phi^{-is}\circ \alpha \circ \Phi^{is}\right) (n)
    \]
    is a $D_s$-module homomorphism. Moreover $\tilde\alpha$ is a projection with image $Q$. This gives the $D_s$-module decomposition $Q\oplus \ker(\tilde\alpha)$ of $N\pua^{s}_{t}$.
\end{proof}

The above contains a weak form of the  analog of Clifford's theorem \cite[Theorem 23.1]{lassueur2020modular}. Here is part of the strong form \cite[Theorem 23.3]{lassueur2020modular}:

\begin{proposition}
\label{prop:clifford}
    Let $M\pda^{s}_{t}=M_{1}^{n_{1}}\oplus\dots\oplus M_{r}^{n_{r}}$ be the decomposition of an irreducible $D_s$-module $M$ into non-isomorphic, nonzero $D_{t}$-modules $M_{i}$. Then $\Phi^s$ permutes the $M_{i}^{n_{i}}$ transitively. In particular, $n_{1}=\cdots=n_{r}$ and $\dim_{F} M_{1}=\cdots=\dim_{F} M_{r}$.
\end{proposition}
\begin{proof}
    The group $\langle\Phi^s\rangle=\{\dots,\Phi^{-s},1,\Phi^s,\Phi^{2s},\dots\}$ permutes the $M_{i}^{n_{i}}$ by the uniqueness of the isotypic decomposition. 
    This action is transitive because $M$ is irreducible.
    Thus $\dim_{F} M_{i}^{n_{i}}=n_{i}\cdot\dim_{F} M_{i}$ are all equal.
    Since $\Phi^s(M_i^{n_i}) = [\Phi^s( M_i)]^{n_i}$, $\langle\Phi^s\rangle$ also permutes the isomorphism classes of the $M_{i}$ transitively. So $\dim_{F} M_{i}$ are all equal.
\end{proof}
Here is an analog of \cite[Section 7.4]{serre1977linear}.

\begin{theorem}[Mackey's irreducibility criterion]
\label{thm:mackey-irreducibility-criterion}
    Let $N$ be an irreducible $D_t$-module. Then $N\pua^{1}_{t}$ is irreducible if and only if $\Phi^i\otimes N$ is not isomorphic to $N$ as $D_t$-modules for each $i=1,2,\dots,t-1$. 
\end{theorem}
\begin{proof}
We first reformulate both sides of the claim. Since $N\pua^{1}_{t}$ is semisimple by \Cref{prop:semisimple}, it is irreducible precisely when $\End_D(N\pua^{1}_{t}) = C$. And since $N$ and $\Phi^i\otimes N$ are irreducible $D_t$-modules, $\Phi^{i}\otimes N$ is not isomorphic to $N$ precisely when $\Hom_{D_t}(N,\Phi^i\otimes N)=0$ (Schur's lemma).
This reduces the theorem to showing that 
$\End_D(N\pua^1_t)=C$ if and only if $\Hom_{D_t}(N,\Phi^i\otimes N)=0$ for $i=1,2,\dots,t-1$.

Tensor-hom adjunction (or Frobenius reciprocity) gives
    \begin{equation*}
        \begin{split}
            \End_{D}(N\pua^{1}_{t}) 
            & \cong \Hom_{D_t}(N,N\pua^{1}_{t}\pda^{1}_{t}) 
            = \bigoplus_{i=0}^{t-1} \Hom_{D_t}(N,\Phi^i\otimes N).
        \end{split}
    \end{equation*}

The $i=0$ summand on the right side is $\Hom_{D_t}(N,N)=C$, and thus $\End_D(N\pua^1_t)=C$ if and only if $\Hom_{D_t}(N,\Phi^i\otimes N)=0$ for each $i=1,2,\dots,n$.

\end{proof}

The above describes (ir)reducibility of induced modules. For restricted modules, we have the following theorem.

\begin{theorem}
\label{thm:absolute-irreducibility}
If a $D$-module $M$ is irreducible and $M\pda^{1}_{t}$ is reducible,
then $M \cong N\pua^{1}_{s}$ for some divisor $s>1$ of $\gcd(t,\dim_FM)$ and some irreducible $D_{s}$-module $N$.
\end{theorem}

\begin{proof}
Pick a minimal divisor $s>1$ of $t$ for which $M\pda^{1}_{s}$ is reducible and 
let $N\subsetneq M\pda^{1}_{s}$ be an irreducible $D_{s}$-submodule.
Since $M$ is irreducible, the map 
\[
    N\pua^{1}_{s}\to M\colon \quad \Phi^{i}\otimes n\mapsto \Phi^{i}(n)
\]
is surjective. 
To finish the proof, it suffices to show that $N\pua^{1}_{s}$ is irreducible, for then the above map is an isomorphism. It also follows that $s$ divides $\dim_{F}M$ since $\dim_{F}N\pua^{1}_{s}=s\cdot\dim_{F}N$.

Suppose $N\pua^{1}_{s}$ is reducible. 
By \Cref{thm:mackey-irreducibility-criterion}, 
we have $N\cong \Phi^{i}(N)$ as $D_{s}$-modules for some $0 < i < s$. Using $\Phi^s(N)=N$ and iterating the isomorphism $N\cong \Phi^{i}(N) \cong \Phi^{2i}(N)\cong \cdots$ give a $D_{s}$-module isomorphism 
\[
    \psi_{0}\colon N\to \Phi^{r}(N)
\]
for some $r$ dividing $s$ with $0<r<s$. 
Let 
\[
    \psi_i=\Phi^{ir}\circ\psi_0\circ\Phi^{-ir} \colon \Phi^{ir}(N) \rightarrow \Phi^{(i+1)r}(N).
\]
Then $c\coloneqq\psi_{s/r-1}\circ\ldots\circ\psi_{1}\circ\psi_0$ is in $\End_{D_{s}}(N) = C$. 
After rescaling each $\psi_i$ by ${c}^{-r/s}$, we may assume $c=1$. We further define 
\[
    \Psi_i = \psi_{i-1}\circ\cdots\circ\psi_0 \colon N \to \Phi^{ir}(N)
\]
and
\[
\Psi\colon N \to M\colon \quad n\mapsto \sum_{i=0}^{s/r-1} \Psi_i(n).
\]
Since $\Psi_1 =\psi_0\colon N\to \Phi^{r}(N)$ is an isomorphism, 
if $n\in N$ then $\Phi^r(n)=\Psi_1(n')$
for some $n'\in N$. 
From $\Psi_{s/r}=c=1=\Psi_{0}$ we get
\[
    \Phi^r\left(\Psi(n)\right) =  \sum_{i=0}^{s/r-1} \Phi^r \Psi_i(n) = \sum_{i=0}^{s/r-1} \Psi_{i+1}(n') = \Psi(n')
\]
which shows that $\Psi(N) \subsetneq M\pda^1_r$ is a $D_r$-module. 
Thus $M\pda^1_r$ is reducible, contradicting the minimality of $s$.
\end{proof}

Repeated use of \Cref{thm:absolute-irreducibility} implies that an irreducible $D$-module is induced from an absolutely irreducible module.
\cite{bou2024solving} uses the corollary below to detect absolute irreducibility.

\begin{corollary}
\label{cor:absolute-irreducibility}
    If a $D$-module $M$ is irreducible and $M\pda^{1}_{t}$ is reducible,
    then $M\pda^1_p$ is reducible for some prime divisor $p$ of $\gcd(t,\dim_{F} M)$.
\end{corollary}
\begin{proof}
    By \Cref{thm:absolute-irreducibility}, $M\cong N\pua_s^1$ for some $D_s$-module $N$ and some divisor $s>1$ of $\gcd(t,\dim_{F}M)$.
    Let $p$ be a prime divisor of $s$. By \Cref{remark:basic-properties}, $M\pda^1_p\cong N\pua^1_s\pda^1_p\cong \bigoplus_{i=0}^{p-1} \left(\Phi^{i}\otimes N\pua^p_s\right)$ is reducible.
\end{proof}

\section{Main theorem}
\label{section:main}
We now prove our main result by reducing to the essential case (\Cref{lemma:first-classification}) where the total Picard-Vessiot ring is a field. 

\begin{theorem}
\label{thm:general-classification}
Let $(F,\phi)= (C (x), x\mapsto x+1)$.
Let $M$ be a $D$-module of dimension $3$. 
If a nonzero $\mathcal{S}$-point of $M$ is $2$-expressible, then
\begin{enumerate}
    \item\label{general-case-1} $M$ is reducible over $F$, or
    \item\label{general-case-2} $M \cong N\pua^{1}_{3}$ for some $D_{3}$-module $N$, or
    \item\label{general-case-3} $M\cong S^{2}(N)\otimes P$ for some $D$-modules $N$ and $P$ of dimensions $2$ and $1$, respectively.
\end{enumerate}
\end{theorem}

We can now deduce the restatement of 
\Cref{thm:general-classification} in \Cref{section:introduction}: 
let $L$ be a $2$-solvable operator of order 3, so that \Cref{thm:general-classification} applies to $M_{L}$. If $M_{L}$ is reducible, so is $L$. 
If $M\cong N\pua^{1}_{3}$, then $L$ is gauge equivalent to some $\Phi^{3}+f$ per the discussion in \Cref{section:restriction-induction}. 
Finally suppose that $M_{L}\cong S^{2}(N)\otimes P$ and $L_{2}$ and $L_{1}$ are operators associated to $N$ and $P$, respectively. 
If $L_{2}^{\cs 2}$ has order $2$, then $L$ is reducible. 
Otherwise $L_{2}^{\cs 2}$ has order $3$ and $L$ is gauge equivalent to $L_{2}^{\cs 2}\cs L_{1}$.

\begin{proof}[Proof of \Cref{thm:general-classification}]
    If $M$ is reducible, \eqref{general-case-1} holds. 
    If $M$ is irreducible but not absolutely irreducible,
    then $M \cong N\pua^{1}_{3}$ by \Cref{thm:absolute-irreducibility}.
    
    We may now assume $M$ is absolutely irreducible.
    Since $L$ is $2$-solvable, $\Gal(L)$ is Eulerian
    by \Cref{thm:eulerian-group}. 
    By \Cref{prop:facts} \eqref{fact-2}, there exists $t\ge 1$ for which the total Picard-Vessiot ring of $M\pda^{1}_{t}$ is a difference field $(K,\phi^{t})$ and $\Gal(M\pda^{1}_{t})=\Gal(M)^{\circ}$.
    Thus $\Gal (M\pda_{t}^{1})$ is Eulerian and connected.
    
    Let $Q_{1}\oplus Q_{2}\oplus\dots\oplus Q_{r}$ be an irreducible decomposition of $S^{2}(M)$. Then
    \[
        S^{2}(M\pda_{t}^{1}) = Q_{1}\pda^{1}_{t}\oplus Q_{2}\pda^{1}_{t}\oplus \cdots\oplus Q_{r}\pda^{1}_{t}.
    \]
    By \Cref{prop:clifford}, $Q_{i}\pda^{1}_{t}$ is a direct sum of irreducible $D_{t}$-modules of the same dimension. 
    If we now apply \Cref{lemma:first-classification} below to the difference module $M\pda^1_t$, we must have $r=2$, $\dim_{F} Q_{1}\pda^{1}_{t} = 1$, and $\dim_{F} Q_{2}\pda^{1}_{t} = 5$. 
    In particular, $\dim_{F} Q_{1}=1$.
    The Symmetric Square Criterion (\Cref{prop:sym-square}) then gives $M\cong S^{2}(N)\otimes P$ for some $D$-modules $N$ and $P$ of dimensions $2$ and $1$, respectively.
\end{proof}

To finish the proof of \Cref{thm:general-classification}, we need

\begin{lemma}
\label{lemma:first-classification}
    Let $(F,\phi)=(C (x),x\mapsto x+1)$. 
    Let $M$ be an absolutely irreducible $D$-module of dimension $3$ whose total Picard-Vessiot ring is a field and whose Galois group is Eulerian and connected. 
    Then $S^2(M)$ is the direct sum of two irreducible $D$-modules of dimensions $1$ and $5$, respectively.
\end{lemma}
\begin{proof}
$G\coloneqq \Gal(M/F)$ embeds into $\GL_3(C)$ because $\dim_{F} M=3$.
Since $G$ is connected, the group homomorphism 
\[
    G\hookrightarrow \GL_{3}(C)\xrightarrow{\det} \mathbb{G}_{m}
\]
has image $1$ or $\mathbb G_m$ and a kernel $H\le \SL_{3}(C)$.
Let $K$ be the total Picard-Vessiot ring of $M$ over $F$. We further let $K'= K^{H^{\circ}}$ and $K''= K^{H}$.
Then $K/F$ decomposes as
\begin{center}
    \begin{tikzcd}
    K \arrow[d,-, "H^{\circ}"'] \arrow[dd,-, "H\le \SL_3(\mathbb{C})", bend left=40] \arrow[ddd,-, "G"', bend right=70]\\
    K' \arrow[d,-, "\text{finite}"'] \\
    K'' \arrow[d,-, "1\text{ or }\mathbb{G}_m"] \\
    F
    \end{tikzcd}
\end{center}

First we show $K'/F$ is a total Picard-Vessiot extension with Galois group $G/H^{\circ}$ and $G/H^{\circ}$ is either $1$ or $\mathbb{G}_{m}$.
Since $H^{\circ}$ is characteristic in $H$ and $H\trianglelefteq G$, we have $H^{\circ}\trianglelefteq G$ and $K'/F$ is a total Picard-Vessiot extension with Galois group $G/H^{\circ}$. As a quotient of $G$, $G/H^{\circ}$ is connected. 
If $G/H=1$, then $G/H^{\circ}$ is finite and connected, hence $1$. 
If instead $G/H= \mathbb{G}_{m}$, then $G/H^{\circ}$ is connected of dimension $1$ with quotient $\mathbb{G}_{m}$. 
The only connected algebraic groups of dimension $1$ are $\mathbb{G}_{a}$ and $\mathbb{G}_{m}$ \cite[Theorem 3.4.9]{springer1998linear}, so $G/H^{\circ}= \mathbb{G}_{m}$.

Next, note that $H^\circ \le H\le SL_3(\mathbb{C})$. Since $H^\circ$ is a closed subgroup of the Eulerian group $G$, $H^\circ$ is again Eulerian. By \cite[Proposition 4.1]{singer1985solving}, the possibilities for a group like $H^{\circ}$ that is both Eulerian and in $\SL_{3}(\mathbb{C})$ are:
\begin{enumerate}[label=(\alph*)]
    \item\label{case-a} $H^{\circ}$ is finite.
    \item\label{case-b} $V(M)$ has an $H^{\circ}$-submodule of dimension $2$.
    \item\label{case-c} $V(M)$ has an $H^{\circ}$-submodule of dimension $1$.
    \item\label{case-d} $H^{\circ}\cong \PSL_{2}(C)$, and $V(M)\cong R_{2}$ as $\PSL_{2}(C)$-modules. Here $R_{2}$ is the degree $2$ homogeneous part of $C[x,y]$, equipped with the standard action of $\SL_2(C)$ on degree $1$ and extended to all degrees. Also the $\PSL_{2}(C)$-action on $V(M)$ is via the isomorphism $H^{\circ}\cong \PSL_{2}(C)$.
\end{enumerate}

Since $M$ is absolutely irreducible, the contrapositive of \Cref{prop:facts}\eqref{fact-1} shows that $G$ is not solvable.
To conclude the proof, we will show that Cases \ref{case-a}--\ref{case-c} force $H^{\circ}$ (and hence $G$) to be solvable and so cannot occur. 
We will show that Case \ref{case-d} gives our desired result.

\textbf{Case \ref{case-a}:}
We must have $H^{\circ}=1$ (solvable). This rules out Case \ref{case-a}.

\textbf{Case \ref{case-b}:}
Since $M$ is irreducible, an $H^{\circ}$-submodule $U\subset V(M)$ of dimension $2$ cannot be $G$-stable in light of \Cref{prop:solution-space-correspondence}. 
Thus there exists $g\in G$ such that $g(U)\neq U$, and $g(U)$ is an $H^{\circ}$-submodule because $H^{\circ}\trianglelefteq G$. A dimension count gives $\dim_{F}(U\cap g(U))=1$, reducing to Case \ref{case-c}.

\textbf{Case \ref{case-c}:}
Let $U\subset V(M)$ be a $H^{\circ}$-submodule of dimension $1$.
The restriction of the irreducible $G$-representation $V(M)$ to $H^{\circ}\trianglelefteq G$ is a semisimple representation $V(M)=g_1(U)\oplus g_2(U)\oplus g_3(U)$ over some $g_1,g_2,g_3\in G$. 
Since each $g_i(U)$ is $1$-dimensional, $\Gal(M/K')$ is a subgroup of
$\mathbb{G}_{m}^{3}$. But $K$ is the total Picard-Vessiot field of $M$ over $F$ (and hence over $K'$ too), so 
$$H^{\circ}=\Gal(K/K')=\Gal(M/K')\le \mathbb{G}_{m}^{3}.$$
Therefore $H^{\circ}$ is a solvable group. This rules out Case \ref{case-c}.

\textbf{Case \ref{case-d}:}
From $V(M)\cong R_{2}$, we get
\[
    V(S^{2}(M))\cong S^{2}(V(M))\cong S^{2}(R_{2}).
\] 
The $\PSL_{2}(C)$-module $S^{2}(R_{2})$ is known to be a direct sum of irreducible modules of dimensions $1$ and $5$. 
Thus $V(S^{2}(M))$ decomposes in the same way into some irreducible $H^{\circ}$-modules, say $W_1$ and $W_2$, of dimensions $1$ and $5$ respectively. Since the dimensions of $W_1$ and $W_2$ are different and since these are the \emph{only} irreducible $H^{\circ}$-submodules of $V(S^{2}(M))$, the irreducible $H^{\circ}$-submodule $g(W_i)$ must equal $W_i$ within $V(S^{2}(M))$, for each $g\in G$ and $i=1,2$. In other words, the $W_i$ are $G$-stable and $V(S^{2}(M))=W_1\oplus W_2$ is a decomposition of $G$-modules.
By \Cref{prop:solution-space-correspondence}, $S^{2}(M)$ has corresponding irreducible $D$-submodules of dimensions $1$ and $5$. This produces the desired direct sum decomposition of $S^{2}(M)$.
\end{proof}

\appendix

\section{A Gr\"obner basis computation}
\label{appendix:groebner}

We now verify a computation needed in the proof of \Cref{prop:sym-square}. Namely, we show the relations in \eqref{eq:g-a-relations} are contained in the ideal generated by the relations \eqref{eq:g-relations} and $1-\det(g_{ij})\cdot z$ (the latter relation is used to ensure that $\det(g_{ij})$ is nonzero by the Rabinowitsch trick). This verification is done using Gröbner bases in the next three steps and the Python code that follows.
\begin{enumerate}
    \item First, we compute a Gröbner basis \texttt{I\_groebner} of the ideal \texttt{I} generated by the relations \eqref{eq:g-relations} and $1-\det(g_{ij})\cdot z$.
    \item Second, we augment this Gröbner basis by the scaled relations \texttt{J} coming from \eqref{eq:g-a-relations}. We call the resulting list \texttt{IJ}. In more details, we let \texttt{G} and \texttt{A} store the entries on the left and right sides of \eqref{eq:g-a-relations}, respectively.
    We let \texttt{J} store the relations \eqref{eq:g-a-relations} after scaling by $4g_{11}g_{21}^{2}$ to clear denominators. The list \texttt{IJ} concatenates the lists \texttt{J} and \texttt{I\_groebner}.
    \item Finally, we see that the Gröbner basis \texttt{IJ\_groebner} of the augmented list \texttt{IJ} is identical to the Gröbner basis \texttt{I\_groebner} of \texttt{I} since line \texttt{17} returns \texttt{True}. Therefore the relations in \texttt{J} are already generated by the relations in \texttt{I}.
\end{enumerate}

\definecolor{bg}{rgb}{0.95,0.95,0.95}
\usemintedstyle{trac}
\begin{minted}[breaklines,bgcolor=bg,fontsize=\footnotesize,linenos]{python}
import sympy
from sympy import symbols, groebner
from sympy.matrices import Matrix

a11, a12, a21, a22, g11, g12, g13, g21, g22, g23, g31, g32, g33, z = symbols('a11 a12 a21 a22 g11 g12 g13 g21 g22 g23 g31 g32 g33 z')

G = Matrix([[g11, g12, g13], [g21, g22, g23], [g31, g32, g33]]) 
A = g11 * Matrix([[a11**2, 2*a11*a12, a12**2], [a11*a21, a11*a22+a12*a21, a12*a22], [a21**2, 2*a21*a22, a22**2]])

I = [1-G.det()*z, g21**2-g11*g31, g23**2-g13*g33, 2*g21*g22-g11*g32-g31*g12, 2*g22*g23-g12*g33-g32*g13, g22**2-g12*g32 + 2*g21*g23-g11*g33-g31*g13]
I_groebner = groebner(I, g11, g12, g13, g21, g22, g23, g31, g32, g33, z)

J = (G - A.subs([(a11, 1), (a12, g12/(2*g11)), (a21, g21/g11), (a22, g32/(2*g21))])) * 4 * g11 * g21**2
IJ = J[:][:] + I_groebner[:]
IJ_groebner = groebner(IJ, g11, g12, g13, g21, g22, g23, g31, g32, g33, z)

IJ_groebner == I_groebner # returns True
\end{minted}

\section{Mackey's formula}
\label{appendix:mackey-formula}
In this appendix, we refine the decomposition in \Cref{prop:clifford} in analogy to representation theory.

\begin{theorem}[Mackey's formula]
\label{thm:mackey-formula}
    Let $s,t\ge 1$, $d=\gcd(s,t)$, and $l=\lcm(s,t)$. Given a $D_s$-module $N$, we have a $D_t$-module isomorphism
    \begin{equation}\label{eq:mackey}
        N\pua^{1}_{s}\pda^{1}_{t}
        \cong 
        \bigoplus_{i=0}^{d-1} 
        \left( 
        \Phi^{i}\otimes N
        \right)\pda^{s}_{l}\pua^{t}_{l}.
    \end{equation}
\end{theorem}
\begin{proof}
    Consider the isomorphisms
    \[
        D\cong \bigoplus_{i=0}^{d-1} D_{d}\otimes \Phi^{i}
        \cong \bigoplus_{i=0}^{d-1} D_{t}\otimes_{D_{l}}(\Phi^{i}\otimes D_{s})
    \]
    where the second isomorphism sends $\Phi^{jt+ks}\otimes \Phi^{i}$ to $\Phi^{jt}\otimes \Phi^{i}\otimes \Phi^{ks}$. 
    The first and third term are isomorphic as $(D_{t},D_{s})$-bimodules.
    Therefore
    \begin{equation*}
        \begin{split}
            N\pua^{1}_{s}\pda^{1}_{t} 
            & \cong \left(D\otimes_{D_{s}}N\right)\pda^{1}_{t}
            \cong \bigoplus_{i=0}^{d-1} D_{t}\otimes_{D_{l}}\left(\Phi^{i}\otimes D_{s}\right)\otimes_{D_{s}}N\\
            & \cong \bigoplus_{i=0}^{d-1} D_{t}\otimes_{D_{l}}\left(\Phi^{i}\otimes N\right)\pda^{s}_{l}
            \cong \bigoplus_{i=0}^{d-1} 
            \left(\Phi^{i}\otimes N\right)\pda^{s}_{l}\pua^{t}_{l}.
        \end{split}
    \end{equation*}
\end{proof}

\begin{proposition}
\label{prop:mackey}
    If $N$ is absolutely irreducible, 
    then the direct summands 
    $( \Phi^{i}\otimes N)\pda^{s}_{l}\pua^{t}_{l}$
    in \eqref{eq:mackey} are irreducible and induced from the absolutely irreducible modules 
    $(\Phi^{i}\otimes N)\pda^{s}_{l}$.
\end{proposition}
\begin{proof}
    Since $N$ is absolutely irreducible, so is $\Phi^{i}\otimes N$, and by the definition of absolute irreducibility, $(\Phi^{i}\otimes N)\pda^{s}_{l}$ as well. 
    
    To see that $\tilde{N}\coloneqq (\Phi^{i}\otimes N)\pda^{s}_{l}\pua^{t}_{l}$ is irreducible, let $P$ be a nonzero irreducible $D_t$-submodule of $\tilde{N}$. We must show that $P=\tilde N$. To do this, we first consider the restriction $P\pda^t_l$, which is a $D_l$-submodule of the restriction $\tilde{N}\pda^t_l$. Using \eqref{eq:basic:3} of \Cref{remark:basic-properties} we can decompose $\tilde{N}\pda^t_l$ as
    \begin{multline}\label{eq:mackey-corollary}
    \tilde{N}\pda^t_l=(\Phi^{i}\otimes N)\pda^{s}_{l}\,\pua^{t}_{l}\,\pda^t_l
    = 
    (\Phi^{i}\otimes N)\pda^{s}_{l}
    \oplus
    (\Phi^{i+t}\otimes N)\pda^{s}_{l}\\
    \oplus
    (\Phi^{i+2t}\otimes N)\pda^{s}_{l}
    \oplus\cdots\oplus
    (\Phi^{i+(l-t)}\otimes N)\pda^{s}_{l}
    \end{multline}
    Since we just proved that each such direct summand in \eqref{eq:mackey-corollary} is absolutely irreducible (hence irreducible), at least one of these direct summands must lie within $P\pda^t_l$. But $P\pda^t_l$ is just $P$ viewed as a $D_l$-module, so $P$ will contain the same direct summand and additionally all its $\Phi^t$-translates. But $\Phi^t$ acts transitively on the direct summands of \eqref{eq:mackey-corollary}, so $P$ contains all the direct summands of \eqref{eq:mackey-corollary}. In other words, $P=\tilde{N}$, and so $\tilde{N}$ is irreducible.
\end{proof}

\section{More results on absolute irreducibility}
\label{appendix:absolute irreducibility}

The following theorem clarifies why our definition of absolute irreducibility agrees with other common usages of the term.

\begin{theorem}\label{G^0-mod} 
    Assume $(F,\phi)=(C(x),x\mapsto x+1)$.
    For $L \in D$ and $G=\Gal(L)$, the following are equivalent.
    \begin{enumerate}
        \item\label{item:G0mod-1} $V(L)$ is an irreducible $H$-module for every algebraic subgroup $H$ of $G$ with $[G:H]$ finite.
        \item\label{item:G0mod-2} $V(L)$ is an irreducible $G^\circ$-module, where $G^\circ$ is the connected component of the identity in $G$.
        \item\label{item:G0mod-3} $L$ is absolutely irreducible. 
    \end{enumerate}
\end{theorem}

\begin{proof}
    We may assume that $L$ is irreducible, otherwise \eqref{item:G0mod-1}--\eqref{item:G0mod-3} are all false.
    (By \Cref{prop:solution-space-correspondence}, 
    $V(L)$ is a reducible $G$-module iff
    $L$ is reducible over $F$.)
    
    \eqref{item:G0mod-1}$\Leftrightarrow$\eqref{item:G0mod-2}:
    Clearly \eqref{item:G0mod-2} is the special case $H=G^{\circ}$ of \eqref{item:G0mod-1}. 
    Conversely, let $H$ be a finite-index, closed subgroup of $G$. By \cite[Proposition 2.2.1]{springer1998linear}, $H$ contains $G^\circ$.
    So if $V(L)$ is irreducible as a $G^{\circ}$-module, then it is also irreducible as an $H$-module.
    
    \eqref{item:G0mod-1}$\Rightarrow$\eqref{item:G0mod-3}: 
    We will prove the contrapositive of \eqref{item:G0mod-1}$\Rightarrow$\eqref{item:G0mod-3}. 
    Suppose that $L$ is not absolutely irreducible. 
    By \Cref{thm:absolute-irreducibility}, $M_L\cong N\pua^1_s$ for some $s>1$. 
    By the discussion after \Cref{remark:basic-logic}, $L$ is gauge equivalent to some $L'\in D_s$.

    Since replacing $L$ by $L'$ does not alter the truth of the implication \eqref{item:G0mod-1}$\Rightarrow$\eqref{item:G0mod-3}, we can assume that $L$ is in $D_s$.
    So now $L$ only contains powers of $\Phi^s$, and we can write
    \[ V(L) = V(L)e_0 \oplus \cdots \oplus V(L)e_{s-1}\]
    for some orthogonal idempotents $e_i$ within the Picard-Vessiot ring $R$ of $\LCLM(L, \Phi^s - 1)$, with $(\Phi^s-1)e_i=0$ for $i=0,\dots,s-1$.

    We first argue that $\Gal(R)$ permutes $\{e_0,\dots,e_{s-1}\}$.
    Since $\Phi^s-1$ has coefficients over $F$, $\Gal(R)$ leaves $V(\Phi^s-1)$ invariant.
    Since $\Gal(R)$ acts as difference ring homomorphisms on $R$, an orthogonal set of $s$-many idempotents within $V(\Phi^s-1)$ must be sent to another orthogonal set of $s$-many idempotents within $V(\Phi^s-1)$. But the only such set is $\{e_0,\dots,e_{s-1}\}$, and so $\Gal(R)$ permutes it.
    
    Thus $\Gal(R)$ also acts by permutation on $\{V(L)e_0, \dots, V(L)e_{s-1}\}$, and the stabilizer $K$ of $V(L)e_0$ is a finite-index subgroup of $\Gal(R)$. 
    Consequently, $K\cap G$ is a finite-index subgroup of $G$ that continues to stabilize $V(L)e_0$. Therefore \eqref{item:G0mod-1} fails by taking $H=K\cap G$.

    \eqref{item:G0mod-3}$\Rightarrow$\eqref{item:G0mod-2}:
    Let $M = D/DL$.
    By \Cref{prop:facts} (2), there exists $t$ for which $\Gal(M\pda^{1}_{t}) = G^{\circ}$.
    If $V(L)$ is a reducible $G^{\circ}$-module, then the same is true for $V(M\pda^{1}_{t})$. Then
    $M\pda^{1}_{t}$ is reducible, and $M$ is not absolutely irreducible.
\end{proof}

\bibliographystyle{plain}

\begin{thebibliography}{10}

\bibitem{amano2006relative}
Katsutoshi Amano, \emph{{Relative invariants, difference equations, and the Picard-Vessiot theory}}, 2006. Available online at: \url{https://arxiv.org/abs/math/0503291}.

\bibitem{bou2024solving}
Heba Bou~KaedBey, Mark van Hoeij, and Man~Cheung Tsui, \emph{Solving Third Order Linear Difference Equations in Terms of Second Order Equations}, 2024, In Proceedings of the 2024 International Symposium on Symbolic and Algebraic Computation (ISSAC '24). Association for Computing Machinery, New York, NY, USA, 457–463. \url{https://doi.org/10.1145/3666000.3669719}.

\bibitem{hendricks1999solving}
Peter~A Hendricks and Michael~F. Singer, \emph{Solving difference equations in finite terms}, Journal of Symbolic Computation \textbf{27} (1999), no.~3, 239--259.

\bibitem{hoeij2007solving}
Mark van Hoeij, \emph{Solving third order linear differential equations in terms of second order equations}, Proceedings of the 2007 International Symposium on Symbolic and Algebraic Computation, 2007, pp.~355--360.

\bibitem{hoeij2006descent}
Mark van Hoeij and Marius van~der Put, \emph{Descent for differential modules and skew fields}, Journal of Algebra \textbf{296} (2006), no.~1, 18--55.

\bibitem{kaplansky1957introduction}
I.~Kaplansky, \emph{An introduction to differential algebra}, Actualit{\'e}s scientifiques et industrielles, Hermann, 1957.

\bibitem{kovacic1986algorithm}
Jerald~J Kovacic, \emph{An algorithm for solving second order linear homogeneous differential equations}, Journal of Symbolic Computation \textbf{2} (1986), no.~1, 3--43.

\bibitem{lam2005introduction}
Tsit-Yuen Lam, \emph{{Introduction to Quadratic Forms over Fields}}, Graduate Studies in Mathematics, vol.~67, American Mathematical Society, 2005.

\bibitem{lang2012algebra}
Serge Lang, \emph{Algebra}, vol. 211, Springer Science and Business Media, 2012.

\bibitem{lassueur2020modular}
Caroline Lassueur and Niamh Farrell, \emph{Modular representation theory of finite groups}, 2020, \url{https://kluedo.ub.rptu.de/frontdoor/deliver/index/docId/6229/file/ModulareDarstellungstheorie1920.pdf}.

\bibitem{nguyen2009d}
K.~A. Nguyen, \emph{On $d$-solvability for linear differential equations}, Journal of Symbolic Computation \textbf{44} (2009), no.~5, 421--434.

\bibitem{nguyen2010solving}
K.~A. Nguyen and M.~van~der Put, \emph{Solving linear differential equations}, Pure and Applied Mathematics Quarterly \textbf{6} (2010), no.~1, 173--208.

\bibitem{person2002solving}
Axelle~Claude Person, \emph{Solving homogeneous linear differential equations of order 4 in terms of equations of smaller order}, North Carolina State University, 2002.

\bibitem{petkovvsek1997wilf}
Marko Petkov{\v{s}}ek, Herbert Wilf, and Doron Zeilberger, \emph{{$A=B$}}, CRC Press, Boca Raton, FL, USA, 1996.

\bibitem{put2003galois}
Marius van~der Put and Michael~F. Singer, \emph{{Galois Theory of Linear Differential equations}}, A Series of Comprehensive Studies in Mathematics, vol. 328, Springer-Verlag, 2003.

\bibitem{put2006galois}
Marius van~der Put and Michael~F. Singer, \emph{Galois Theory of Difference Equations}, Lecture Notes in Mathematics, vol. 1666, Springer-Verlag, 2006.

\bibitem{serre1977linear}
Jean-Pierre Serre, \emph{Linear Representations of Finite Groups}, vol.~42, Springer-Verlag, 1977.

\bibitem{singer1985solving}
Michael~F. Singer, \emph{Solving homogeneous linear differential equations in terms of second order linear differential equations}, American Journal of Mathematics \textbf{107} (1985), no.~3, 663--696.

\bibitem{singer1988algebraic}
Michael~F. Singer, \emph{Algebraic relations among solutions of linear differential equations: Fano's theorem}, American Journal of Mathematics \textbf{110} (1988), no.~1, 115--143.

\bibitem{springer1998linear}
Tonny Springer, \emph{Linear Algebraic Groups}, second ed., Progress in Mathematics, vol.~9, Birkh\"{a}user, 1998.

\bibitem{algo} Hoeij, M. and Bou KaedBey, H. Implementation and Examples \url{https://www.math.fsu.edu/~hboukaed/Implementations}. (Florida State University,2024), \url{https://www.math.fsu.edu/~hoeij/AbsFactor}

\bibitem{hessinger2001computing} Sabrina A. Hessinger, ``Computing the Galois group of a linear differential equation of order four." Applicable Algebra in Engineering, Communication and Computing 11 (2001): 489-536.

\bibitem{cha2010solving} Yongjae Cha, Mark Van Hoeij, and Giles Levy. ``Solving recurrence relations using local invariants." Proceedings of the 2010 International Symposium on Symbolic and Algebraic Computation. 2010.

\end{thebibliography}
 \newcommand{\noop}[1]{}

\end{document}